\documentclass{article}
\usepackage{graphicx}
\usepackage{lineno}
\usepackage[english]{babel}
\usepackage{indentfirst}
\usepackage{amsmath}
\usepackage{amsthm}

\def\qed{\hfill \rule{2.5mm}{2.5mm}}
\newtheorem{thm}{Theorem}[section]
\newtheorem{lem}[thm]{Lemma}

\newcommand{\D}{\displaystyle}
\newcommand{\DF}[2]{\frac{\D#1}{\D#2}}
\newcommand{\lf}{\left\lfloor}
\newcommand{\rf}{\right\rfloor}
\newcommand{\M}{\lf\DF{M+2d-1}{2d}\rf}

\large\normalsize

\title{List ($d$,1)-total labelling of graphs embedded in surfaces\thanks{This work was supported by IIFSDU(2009hw001), NNSF(61070230, 11026184, 10901097) and
RFDP(200804220001, 20100131120017) and SRF for ROCS.}}
\author{Yong Yu, Xin Zhang, Guizhen Liu \thanks{Corresponding author.\quad yuyong6834@yahoo.com.cn, gzliu@sdu.edu.cn}
\\School of Mathematics, Shandong University, Jinan, 250100, China}
\date{}

\begin{document}
\maketitle

\begin{abstract}
The ($d$,1)-total labelling of graphs was introduced by Havet and
Yu. In this paper, we consider the list version of ($d$,1)-total
labelling of graphs. Let $G$ be a graph embedded in a surface with
Euler characteristic $\varepsilon$ whose maximum degree $\Delta(G)$
is sufficiently large. We prove that the ($d$,1)-total choosability
$C_{d,1}^T(G)$ of $G$ is at most $\Delta(G)+2d$.

\noindent {\bf Keywords:} ($d$,1)-total labelling; list
($d$,1)-total labelling; ($d$,1)-total choosability; graphs

\noindent \textbf{MSC:} 05C15
\end{abstract}

\section{Introduction}

In this paper, graph $G$ is a simple connected graph with a finite
vertex set $V(G)$ and a finite edge set $E(G)$. If $X$ is a set, we
usually denote the cardinality of $X$ by $|X|$. Denote the set of
vertices adjacent to $v$ by $N(v)$. The degree of a vertex $v$ in
$G$, denoted by $d_G(v)$, is the number of edges incident with $v$.
We sometimes write $V, E, d(v), \Delta, \delta$ instead of $V(G),
E(G), d_G(v), \Delta(G), \delta(G)$, respectively. Let $G$ be a
plane graph. We always denote by $F(G)$ the face set of $G$. The
degree of a face $f$, denoted by $d(f)$, is the number of edges
incident with it, where cut edge is counted twice. A $k$-, $k^+$-
and $k^-$-vertex ( or face ) in graph $G$ is a vertex ( or face ) of
degree $k$, at least $k$ and at most $k$, respectively.

The ($d$,1)-total labelling of graphs was introduced by Havet and Yu
\cite{4}. A \emph{$k$-($d$,1)-total labelling} of a graph $G$ is a
function $c$ from $V(G)\cup E(G)$ to the color set
$\{0,1,\cdots,k\}$ such that $c(u)\neq c(v)$ if $uv\in E(G)$,
$c(e)\neq c(e')$ if $e$ and $e'$ are two adjacent edges, and
$|c(u)-c(e)|\geq d$ if vertex $u$ is incident to the edge $e$. The
minimum $k$ such that $G$ has a $k$-($d$,1)-total labelling is
called the \emph{($d$,1)-total labelling number} and denoted by
$\lambda_d^T(G)$. Readers are referred to \cite{1,3,5,6,7} for
further research.

Suppose that $L(x)$ is a list of colors available to choose for each
element $x\in V(G)\cup E(G)$. If $G$ has a ($d$,1)-total labelling
$c$ such that $c(x)\in L(x)$ for all $x\in V(G)\cup E(G)$, then we
say that $c$ is an $L$-($d$,1)-total labelling of $G$, and $G$ is
$L$-($d$,1)-total labelable (sometimes we also say $G$ is list
($d$,1)-total labelable). Furthermore, if $G$ is $L$-($d$,1)-total
labelable for any $L$ with $|L(x)|=k$ for each $x\in V(G)\cup E(G)$,
we say that $G$ is $k$-($d$,1)-total choosable. The ($d$,1)-total
choosability, denoted by $C_{d,1}^T(G)$, is the minimum $k$ such
that $G$ is $k$-($d$,1)-total choosable. Actually, when $d=1$, the
list (1,1)-total labelling is the well-known list total coloring of
graphs. It is known that for list version of total colorings there
is a list total coloring conjecture (LTCC). Therefore, it is natural
to conjecture that $C_{d,1}^T(G)=\lambda_d^T(G)+1$.
Unfortunately, counterexamples that
$C_{d,1}^T(G)$ is strictly greater than $\lambda_d^T(G)+1$ can be found in \cite{9}.
Although we can not present a conjecture like LTCC, we conjecture that
$C_{d,1}^T(G)\leq \Delta+2d$ for any graph $G$. In \cite{9}, we
studied the list ($d$,1)-total labelling of special graphs such as
paths, trees, stars and outerplanar graphs which lend positive
support to our conjecture.

In this paper, we prove that, for graphs which can be embedded in a
surface with Euler characteristic $\varepsilon$, the conjecture is
still true when the maximum degree is sufficiently large. Our main
results are the following:

\begin{thm}
Let $G$ be a graph embedded in a surface of Euler characteristic
$\varepsilon\leq0$ and $\Delta(G)\geq\DF{d}{2d-1}
\left(10d-8+\sqrt{(10d-2)^2-24(2d-1)\varepsilon}\right)+1$, where
$d\geq2$. Then $C_{d,1}^T(G)\leq\Delta(G)+2d$.
\end{thm}

\begin{thm}
Let $G$ be a graph embedded in a surface of Euler characteristic
$\varepsilon>0$. If $\Delta(G)\geq 5d+2$ where $d\geq2$, then
$C_{d,1}^T(G)\leq \Delta(G)+2d$.
\end{thm}

We prove two conclusions which are slightly stronger than the
theorems above as follows.

\begin{thm}
Let $G$ be a graph embedded in a surface of Euler characteristic
$\varepsilon\leq0$ and let
$M\geq\DF{d}{2d-1}\left(10d-8+\sqrt{(10d-2)^2-24(2d-1)\varepsilon}\right)+1$
where $d\geq2$. If $\Delta(G)\leq M$, then $C_{d,1}^T(G)\leq M+2d$.
In particular, $C_{d,1}^T(G)\leq \Delta(G)+2d$ if $\Delta(G)=M$.
\end{thm}

\begin{thm}
Let $G$ be a graph embedded in a surface of Euler characteristic
$\varepsilon>0$ and let $M\geq 5d+2$ where $d\geq2$. If
$\Delta(G)\leq M$, then $C_{d,1}^T(G)\leq M+2d$. In particular,
$C_{d,1}^T(G)\leq \Delta(G)+2d$ if $\Delta(G)=M$.
\end{thm}

The interesting cases of Theorem 1.3 and Theorem 1.4 are when $M =
\Delta(G)$. Indeed, Theorem 1.3 and Theorem 1.4 are only a technical
strengthening of Theorem 1.1 and Theorem 1.2, respectively. But
without them we would get complications when considering a subgraph
$H\subset G$ such that $\Delta(H)<\Delta(G)$.

In Section 2, we prove some lemmas. In Section 3, we complete our
main proof with discharging method.

\section{Structural properties}
From now on, we will use without distinction the terms \emph{colors}
and \emph{labels}. Let $c$ be a partial list ($d$,1)-total labelling
of $G$. We denote by $A(x)$ the set of colors which are still available for coloring
element $x$ of $G$ with the partial list ($d$,1)-total labelling $c$. Let $G$ be a minimal
counterexample in terms of $|V(G)|+|E(G)|$ to Theorem 1.3 or Theorem
1.4.

\begin{lem}
$G$ is connected.
\end{lem}
\begin{proof}
Suppose that $G$ is not connected. Without loss of generality, let
$G_1$ be one component of $G$ and $G_2=G\backslash G_1$. By the
minimality of $G$, $G_1$ and $G_2$ are both
$(M+2d)$-($d$,1)-total choosable which implies $G$ is
$(M+2d)$-($d$,1)-total choosable, a contradiction.
\end{proof}

\begin{lem}
For each $e=uv\in E(G)$, $d(u)+d(v)\geq M-2d+4$.
\end{lem}
\begin{proof}
If for some $e=uv\in E(G)$, $d(u)+d(v)\leq M-2d+3$. By the
minimality of $G$, $G-e$ is $(M+2d)$-($d$,1)-total choosable.
We denote this coloring by $c$. Since
$|A(e)|\geq M+2d-(d(u)+d(v)-2)-2(2d-1)\geq M+2d-(M-2d+1)-2(2d-1)\geq1$
under the coloring $c$, we can extend $c$ to $G$, a contradiction.
\end{proof}

\begin{lem}
For any edge $e=uv\in E(G)$ with
$\min\{d(u),d(v)\}\leq\lf\DF{M+2d-1}{2d}\rf$, we have
$d(u)+d(v)\geq M+3$.
\end{lem}
\begin{proof}
Suppose there is some $e=uv\in E(G)$ such that
$d(u)\leq\lf\DF{M+2d-1}{2d}\rf$ and $d(u)+d(v)\leq M+2$.
By the minimality of $G$,
$G-e$ is $(M+2d)$-($d$,1)-total choosable. Erase
the color of vertex $u$, and let $c$ be the partial
list ($d$,1)-total labelling with $|L|=M+2d$. Then
$|A(e)|\geq M+2d-(d(u)+d(v)-2)-(2d-1)\geq M+2d-M-(2d-1)\geq1$
which implies that $e$ can be properly colored. Next, for vertex
$u$, $|A(u)|\geq M+2d-(d(u)+(2d-1)d(u))\geq M+2d-(M+2d-1)\geq1$.
Thus we extend the coloring $c$ to $G$, a contradiction.
\end{proof}

\begin{lem}
{\upshape ([2])} A bipartite graph $G$ is edge $f$-choosable where
$f(uv)=\max\{d(u),d(v)\}$ for any $uv\in E(G)$.
\end{lem}

A \emph{$k$-alternator} for some $k$ $(3\leq k\leq\M)$ is a
bipartite subgraph $B(X,Y)$ of graph $G$ such that
$d_B(x)=d_G(x)\leq k$ for each $x\in X$ and $d_B(y)\geq
d_G(y)+k-M-1$ for each $y\in Y$.

The concept of $k$-alternator was first introduced by Borodin,
Kostochka and Woodall \cite{2} and generalized by Wu and Wang
\cite{8}.

\begin{lem}
There is no $k$-alternator $B(X,Y)$ in $G$ for any integer $k$ with
$3\leq k\leq\M$.
\end{lem}
\begin{proof}
Suppose that there exits a $k$-alternator $B(X,Y)$ in $G$.
Obviously, $X$ is an independent set of vertices in graph $G$ by
Lemma 2.3. By the minimality of $G$, we can color all elements of
subgraph $G[V(G)\backslash X]$ from their lists of size $M+2d$.
We denote this partial
list ($d$,1)-total labelling by $c$. Then for each edge $e=xy\in B(X,Y)$,
$|A(e)|\geq M+2d-\left(d_G(y)-d_B(y)+(2d-1)\right)\geq M+2d-\left(M-d_B(y)+(2d-1)\right)\geq d_B(y)$ and
$|A(e)|\geq M+2d-\left(d_G(y)-d_B(y)+(2d-1)\right)\geq M+2d-(M+2d-k)\geq k$ because $B(X,Y)$ is
a $k$-alternator. Therefore,
$|A(e)|\geq\max\{d_B(y),d_B(x)\}$. By Lemma 2.4, it
follows that $E(B(X,Y))$ can be colored properly from their new
color lists. Next, for each vertex $x\in X$,
$|A(x)|\geq M+2d-\left(d(x)+(2d-1)d(x)\right)\geq M+2d-(M+2d-1)\geq1$
because $d_G(x)\leq k\leq\M$. Thus we extend the coloring $c$ to
$G$, a contradiction.
\end{proof}

\begin{lem}
Let $X_k=\{x\in V(G)\ \big|\ d_G(x)\leq k\}$ and $Y_k=\cup_{x\in
X_k}N(x)$ for any integer $k$ with $3\leq k\leq\M$. If
$X_k\neq\emptyset$, then there exists a bipartite subgraph $M_k$ of
$G$ with partite sets $X_k$ and $Y_k$ such that $d_{M_k}(x)=1$ for
each $x\in X_k$ and $d_{M_k}(y)\leq k-2$ for each $y\in Y_k$.
\end{lem}
\begin{proof}
The proof is omitted here as it is similar with the proof of Lemma
2.4 by Wu and Wang in \cite{8}.
\end{proof}

We call $y$ the $k$-master of $x$ if $xy\in M_k$ and $x\in X_k, y\in
Y_k$. By Lemma 2.3, if $uv\in E(G)$ satisfies $d(v)\leq\M$ and
$d(u)=M-i$, then $d(v)\geq M+3-d(u)\geq i+3$. Together with Lemma
2.6, it follows that each $(M-i)$-vertex can be a $j$-master of at
most $j-2$ vertices, where $3\leq i+3\leq j\leq\M$. Each $i$-vertex
has a $j$-master by Lemma 2.6, where $3\leq i\leq j\leq\M$.

\section{Proof of main results}
By our Lemmas above, $G$ has structural properties in the following.
\begin{itemize}
\item[$(C1)$] $G$ is connected;
\item[$(C2)$] for each $e=uv\in E(G)$, $d(u)+d(v)\geq M-2d+4$;
\item[$(C3)$] if $e=uv\in E(G)$ and $\min\{d(u),d(v)\}\leq\lf\DF{M+2d-1}{2d}\rf$,
then $d(u)+d(v)\geq M+3$.
\item[$(C4)$] each $i$-vertex (if exists) has one $j$-master, where $3\leq i\leq j\leq\M$;
\item[$(C5)$] each $(M-i)$-vertex (if exists) can be a $j$-master of at most $j-2$ vertices, where $3\leq i+3\leq
j\leq\M$.
\end{itemize}

\noindent\textit{\textbf{Proof of Theorem 1.3}} Let $G$ be a minimal
counterexample in terms of $|V(G)|+|E(G)|$ to Theorem 1.3. In this
theorem,
$M\geq\DF{d}{2d-1}\left(10d-8+\sqrt{(10d-2)^2-24(2d-1)\varepsilon}\right)+1\geq10d+1$.
Thus $\lf\DF{M+2d-1}{2d}\rf\geq6$. In the following, we apply the discharging method to complete the
proof by a contradiction. At the very beginning, we assign an
initial charge $w(x)=d(x)-6$ for any $x\in V(G)$. By Euler's formula
$|V|-|E|+|F|=\varepsilon$, we have $\sum\limits_{x\in
V}w(x)=\sum\limits_{x\in V}(d(x)-6)=-6\varepsilon-\sum\limits_{x\in
F}(2d(x)-6)\leq -6\varepsilon$.

The discharging rule is as follows.

\textit{(R1)} each $i$-vertex (if exists) receives charge 1 from
each of its $j$-master, where $3\leq i\leq j\leq5$.

If $M\geq\Delta+3$, then $\delta(G)\geq6$. Otherwise, let $uv\in E(G)$ and $d(u)\leq5$. Then $d(u)+d(v)\leq
M-3+5\leq M+2$ and $d(u)\leq\lf\DF{M+2d-1}{2d}\rf$ as
$\lf\DF{M+2d-1}{2d}\rf\geq6$, which is a  contradiction to (C3). This obviously
contradicts the fact $\delta(G)\geq5$ for any planar graph. Proof of
the theorem is completed. Next, we only consider the case
$\Delta\leq M\leq\Delta+2$.

\noindent{\bf\textit{Claim 1.}} $\delta\geq M-\Delta+3$.
\begin{proof}
If there is some $e=uv\in E(G)$ such that $d(v)\leq M-\Delta+2$,
then $d(u)+d(v)\leq\Delta+(M-\Delta+2)\leq M+2$ and
$d(v)\leq5\leq\lf\DF{M+2d-1}{2d}\rf$ as
$\lf\DF{M+2d-1}{2d}\rf\geq6$, a contradiction to (C3).
\end{proof}

Let $v$ be a $k$-vertex of $G$.

$(a)$\quad If $3\leq k\leq5$, then $w'(v)=w(v)+\sum\limits_{k\leq
i\leq5}1=(k-6)+(6-k)=0$ by (C4) and rule (R1);

$(b)$\quad If $6\leq k\leq M-3$, then for all $u\in N(v)$,
$d(u)\geq6$ by (C3). Therefore, $v$ neither receives nor gives any charge by our rule,
which implies that $w'(v)=w(v)=k-6\geq0$;

$(c)$\quad If $M-2\leq k\leq \Delta$.

{\bf Case 1.}\  $M=\Delta+2$. Then $\delta\geq5$ by Claim 1.
For $k=\Delta$,\  $w'(v)\geq w(v)-3=\Delta-9=M-11$ by (C5) and (R1).

{\bf Case 2.}\  $M=\Delta+1$. Then $\delta\geq4$ by Claim 1. For
$k=\Delta-1$,\ $w'(v)\geq w(v)-3=\Delta-1-6-3=M-11$ by (C5) and rule
(R1). For $k=\Delta$,\ $w'(v)\geq w(v)-3-2=\Delta-6-3-2=M-12$ by
(C5) and rule (R1).

{\bf Case 3.}\  $M=\Delta$. Then $\delta(G)\geq3$ by Claim 1. For
$k=\Delta-2$,\ $w'(v)\geq w(v)-3=\Delta-2-6-3=M-11$ by (C5) and rule
(R1). For $k=\Delta-1$,\ $w'(v)\geq w(v)-3-2=\Delta-1-6-3-2=M-12$ by
(C5) and rule (R1). For $k=\Delta$,\ $w'(v)\geq
w(v)-3-2-1=\Delta-6-3-2-1=M-12$ by (C5) and rule (R1).

For all cases above, $w'(v)\geq M-12>0$ for any $d(v)\geq\Delta-2$ as
$M\geq10d+1\geq21$.

Let $X=\{x\in V(G)\big|d_G(x)\leq\lf\DF{M+2d-1}{2d}\rf\}$. By (C3),
$X$ is an independent set of vertices.

\noindent{\bf\textit{Claim 2.}} The number of
$\left(\lf\DF{M+2d-1}{2d}\rf+1\right)^+$-vertex of $G$ is at least
$M-\lf\DF{M+2d-1}{2d}\rf+3$. That is, $|V(G\backslash X)|\geq
M-\lf\DF{M+2d-1}{2d}\rf+3$.
\begin{proof}
Otherwise, let $Y=N_{x\in X}(x)$ and $B=B(X,Y)$ be the induced
bipartite subgraph. For all $y\in Y$, $d_{G\backslash X}(y)\leq
|Y|-1\leq M-\lf\DF{M+2d-1}{2d}\rf+1$. Therefore,
$d_{B}(y)=d_{G}(y)-d_{G\backslash X}(y)\geq
d_{G}(y)+\lf\DF{M+2d-1}{2d}\rf-M-1$, which implies $B$ is a
$\lf\DF{M+2d-1}{2d}\rf$-alternator of $G$, a contradiction to Lemma
2.5.
\end{proof}

Since $M\geq10d+1$, it follows that $M-12>\lf\DF{M+2d-1}{2d}\rf-5$.
Thus, $w'(v)\geq \lf\DF{M+2d-1}{2d}\rf-5$ when $d_G(v)\geq
\lf\DF{M+2d-1}{2d}\rf+1$. Then $\sum\limits_{x\in
V}w(x)=\sum\limits_{x\in
V}w'(x)>(M-\lf\DF{M+2d-1}{2d}\rf+3)(\lf\DF{M+2d-1}{2d}\rf-5)\geq(2d-1)\left(\DF{M-1}{2d}\right)^2-(10d-8)\DF{M-1}{2d}-15\geq-6\varepsilon$
as $M\geq\DF{d}{2d-1}\left(10d-8+\sqrt{(10d-2)^2-24(2d-1)\varepsilon}\right)+1$.
Then this contradiction completes the proof.\qed\\

\noindent\textit{\textbf{Proof of Theorem 1.4}}
Let $G$ be a minimal
counterexample in terms of $|V(G)|+|E(G)|$ to Theorem 1.4. In this
theorem, $M\geq 5d+2$. We define the initial charge function
$w(x):=d(x)-4$ for all element $x\in V\cup F$. By
Euler's formula $|V|-|E|+|F|=\varepsilon$, we have
$\sum\limits_{x\in V\cup F}w(x)=\sum\limits_{v\in V}(d(v)-4)+\sum\limits_{f\in
F}(d(f)-4)=-4\varepsilon<0$.

The transition rules are defined as follows.

\textit{(R1)}\quad Each 3-vertex (if exists) receives charge 1 from
its $3$-master.

\textit{(R2)}\quad Each $k$-vertex with $5\leq k\leq7$ transfer
charge $\DF{k-4}{k}$ to each 3-face that incident with it.

\textit{(R3)}\quad Each $8^+$-vertex transfer charge $\DF{1}{2}$ to
each 3-face that incident with it.

Analogous with Claim 1 in the proof of Theorem 1.3, it is easy to prove that
$\delta(G)\geq3$ when $\Delta=M$ and $\delta(G)\geq4$ otherwise. Let $v$ be a $k$-vertex of $G$.

For $k=3$, then $w'(v)= w(v)+1=3-4+1=0$ since it receives 1 from its
3-master;

For $k=4$, then $w'(v)= w(v)=0$ since we never change the charge by
our rules;

For $5\leq k\leq7$, then $w'(v)\geq w(v)-k\DF{k-4}{k}=0$ by (R2);

For $8\leq k\leq M-1$, then $w'(v)\geq w(v)-k\DF{1}{2}\geq0$ by
(R3);

If $M>\Delta$, then $M-1\geq\Delta$. Thus $w(v)\geq0$ for all $v\in
V(G)$. Otherwise, $\Delta=M$. Then for $k=\Delta$, $w'(v)\geq
w(v)-\DF{1}{2}M-1=\DF{M}{2}-5$ by (C5) and rules (R1), (R3). Since
$M\geq5d+2\geq12$, we have $w'(v)\geq\DF{M}{2}-5>0$.

Let $f$ be a $k$-face of $G$.

If $k\geq4$, then $w'(f)=w(f)\geq0$ since we never change the charge of them by our rules;

If $k=3$, assume that $f=[v_1,v_2,v_3]$ with $d(v_1)\leq d(v_2)\leq
d(v_3)$. It is easy to see $w(f)=-1$. Consider the subcases as
follows.

$(a)$\ Suppose $d(v_1)=3$. Then $M=\Delta$ and
$d(v_2)=d(v_3)=\Delta$ by (C3). Thus,
$w'(f)=w(f)+\DF{1}{2}\times2=0$ by (R3);

$(b)$\ Suppose $d(v_1)=4$. Then $d(v_3)\geq d(v_2)\geq
M-2d+4-d(v_1)\geq3d+2\geq8$ by (C2). Therefore,
$w'(f)=w(f)+\DF{1}{2}\times2=0$ by (R3);

$(c)$\ Suppose $d(v_1)=5$. Then $d(v_3)\geq d(v_2)\geq
M-2d+4-d(v_1)\geq3d+1\geq7$ by (C2). Therefore,
$w'(f)=w(f)+\DF{3}{7}\times2+\DF{1}{5}>0$ by (R2).

$(d)$\ Suppose $d(v_1)=m\geq6$. Then $d(v_3)\geq d(v_2)\geq6$. Therefore, $w'(f)\geq
w(f)+3\times\min\{\DF{m-4}{m},\DF{1}{2}\}=0$ by (R2) and (R3).

Thus, we have $\sum\limits_{x\in V\cup F}w'(x)\geq0$ which is a
contradiction with $\sum\limits_{x\in V\cup
F}w'(x)=\sum\limits_{x\in V\cup F}w(x)<0$.\qed


\end{document}